\documentclass[journal,twoside,web]{ieeecolor}
\usepackage{lcsys}
\def\BibTeX{{\rm B\kern-.05em{\sc i\kern-.025em b}\kern-.08em
    T\kern-.1667em\lower.7ex\hbox{E}\kern-.125emX}}
\usepackage{graphics} 
\usepackage{epsfig} 
\usepackage{amsmath} 
\usepackage{amssymb}  
\usepackage{amsfonts}       
\usepackage{nicefrac}       
\usepackage{microtype}      
\usepackage{graphicx,wrapfig}

\usepackage[normalem]{ulem}
\usepackage{subcaption}

\usepackage{amsthm}

\usepackage{multirow,multicol}
\usepackage{algorithm,algorithmic}
\newtheorem{proposition}{Proposition}
\newtheorem{assumption}{Assumption}
\newtheorem{definition}{Definition}
\newtheorem{theorem}{Theorem}

\newtheorem{lemma}{Lemma}
\newtheorem{remark}{Remark}

\DeclareMathOperator*{\argmin}{\arg\!\min}

\newcommand{\addcite}[0]{\ifthenelse{\boolean{showcomments}}
{\textcolor{purple}{(add cite(s)) }}{}}%

\usepackage{hyperref}
\usepackage{cleveref}
\let\labelindent\relax
\usepackage{enumitem}
\usepackage{booktabs}

\usepackage{cite}
\usepackage{microtype}    

\title{Sequential QCQP for Bilevel Optimization with Line Search}
\author{Sina Sharifi, Erfan Yazdandoost Hamedani, Mahyar Fazlyab
\thanks{S. Sharifi and M. Fazlyab are with the Department of Electrical and Computer Engineering at Johns Hopkins University, Baltimore, MD 21218, USA.
{\tt\small \{sshari12, mahyarfazlyab\}@jhu.edu}}%
\thanks{E. Yazdandoost Hamedani is with the Department of Systems and Industrial Engineering, The University of Arizona, Tucson, AZ 85721, USA
        {\tt\small erfany@arizona.edu}}%
}

\newcommand{\blue}[1]{{\color{black} \noindent #1}}

\begin{document}
\maketitle
\thispagestyle{empty}
\pagestyle{empty}

\begin{abstract}
Bilevel optimization involves a hierarchical structure where one problem is nested within another, leading to complex interdependencies between levels. We propose a single-loop, tuning-free algorithm that guarantees anytime feasibility, i.e., approximate satisfaction of the lower-level optimality condition, while ensuring descent of the upper-level objective. At each iteration, a convex quadratically-constrained quadratic program (QCQP) with a closed-form solution yields the search direction, followed by a backtracking line search inspired by control barrier functions to ensure safe, uniformly positive step sizes. The resulting method is scalable, requires no hyperparameter tuning, and converges under mild local regularity assumptions. \blue{We establish an $\mathcal{O}(1/k)$ ergodic convergence rate in terms of a first-order stationary metric }
and demonstrate the algorithm’s effectiveness on representative bilevel tasks.
\end{abstract}


\section{Introduction}

\IEEEPARstart{I}{n} this paper, we introduce a novel iterative algorithm for solving bilevel optimization (BLO) problems of the form:
\begin{align}\label{eq:BLO}\tag{BLO}
    f^\star \leftarrow 
    \min_{x,y \in \mathbb{R}^{n+m}} \; f(x,y) 
    \text{ s.t.} \ y \in \argmin_{y \in \mathbb{R}^m} \; g(x,y),
\end{align}
where \( f: \mathbb{R}^n \times \mathbb{R}^m \to \mathbb{R} \) and \( g: \mathbb{R}^n \times \mathbb{R}^m \to \mathbb{R} \) denote the upper- and lower-level objective functions, respectively. 
Despite the wide applicability of \eqref{eq:BLO} in various machine learning (e.g., hyperparameter optimization) \cite{franceschi2018bilevel, gao2022value} and control (e.g., trajectory and gait optimization) \cite{landry2019bilevel, sun2021fast, olkin2024bilevel} applications involving hierarchical decision-making, designing reliable and efficient algorithms for solving such problems remains a significant challenge.

\subsection{Related Work}

A popular approach for solving \eqref{eq:BLO} is \emph{hypergradient descent}. Assuming the lower-level problem has a unique solution, we can define the implicit objective \( \ell(x) = f(x, y^\star(x)) \), where \( y^\star(x) = \argmin_{y \in \mathbb{R}^m} \; g(x,y) \). This implicit function can be optimized via gradient descent by differentiating through the solution of the lower-level problem. However, computing \( \nabla \ell(x) \) exactly requires solving for \( y^\star(x) \) and inverting \( \nabla_{yy}^2 g(x,y^\star(x)) \), which are computationally expensive. To mitigate this, many methods estimate \( \nabla \ell(x) \), such as approximate implicit differentiation (AID) \cite{pedregosa2016hyperparameter, domke2012generic, ghadimi2018approximation, ji2021bilevel} and iterative differentiation (ITD) \cite{shaban2019truncated, franceschi2018bilevel, grazzi2020iteration}.
While hypergradient descent has been widely used in applications such as hyperparameter tuning and meta-learning, its effectiveness often hinges on the tractability of the lower-level problem and the stability of the resulting gradient estimates. 
More recently, \cite{grontas2024big} approximated the hypergradient in a distributed manner to improve scalability.

Another class of bilevel solvers reformulates \eqref{eq:BLO} as a single-level constrained optimization problem:
\begin{align}\label{eq:BLO-reformulated-1}
    &\min_{x,\, y} \; f(x,y) \quad  \text{s.t.} \quad  h(x,y)= 0.
\end{align}
Such problems are typically solved using primal-dual or penalty-based methods, e.g., \cite{jiang2024primal, sow2022primal, shen2023penalty, mehra2021penalty, liu2022bome, kwon2023fully}. Two popular reformulations fall under this framework: value function-based approaches, where \( h(x,y) = g(x,y) - \min_y g(x,y) \), and stationary-seeking approaches, where \( h(x,y) = \nabla_y g(x,y) \). While the value function formulation is always equivalent to \eqref{eq:BLO}, it often suffers from non-differentiability and computational burden in evaluating \( \min_y g(x,y) \). In contrast, the stationary-seeking formulation is equivalent to \eqref{eq:BLO} under some assumptions, e.g., when the lower-level objective \( g \) is convex or satisfies the weak Polyak-Łojasiewicz (PL) condition \cite{csiba2017global}. Our proposed framework also falls into this second category.



Building on the recent application of control barrier functions~\cite{ames2016control} to optimization algorithm design (e.g., \cite{schechtman2023orthogonal,allibhoy2023control}), {the work in} \cite{sharifi2025safe} developed a continuous-time framework to solve a relaxed single-level reformulation of \eqref{eq:BLO-reformulated-1}, given by
\begin{align}\label{eq:bilevel-approx-1}
    f^\star_{\varepsilon} \leftarrow  
    \min_{x,y}~f(x,y) \quad 
    \text{s.t.} \quad  h(x,y) - \varepsilon^2 \leq 0,
\end{align}
where \( h(x,y) = \|\nabla_y g(x,y)\|^2 \) from now {on\footnote{Throughout, we denote the Euclidean norm by $\|\cdot\|$.}}, \( \varepsilon > 0 \) is a user-defined tolerance parameter, and {$f^\star_{\varepsilon}\leq f^\star$}. Treating the gradient flow of the upper-level objective as the nominal (``unsafe'') dynamics, \cite{sharifi2025safe} proposed a safety filter by employing \( h \) as a barrier function. The resulting dynamics is obtained by solving the following convex quadratic program (QP):
\begin{alignat}{2} \label{eq: first-order lower-level dynamics 2}
    &\min_{(\dot{x}_d,\dot{y}_d)}~  \frac{1}{2}\|\dot{x}_d+\nabla_x f(x,y)\|^2  + \frac{1}{2}\|\dot{y}_d+\nabla_y f(x,y)\|^2\\
    &\text{s.t.}\ \nabla_x h(x,y)^\top \dot{x}_d \! + \! \nabla_y h(x,y)^\top \dot{y}_d \! + \! \alpha_{b} (h(x,y) \! - \! \varepsilon^2) \leq 0, \notag    
\end{alignat}
where $\alpha_b > 0$ is a barrier parameter. This QP computes the minimal perturbation to the nominal gradient flow required to ensure that {the feasible set defined by the $\varepsilon^2$-sublevel set of \( h \), denoted by \( L_{\varepsilon^2}^{-}(h) \), is forward invariant}. The solution to this QP yields the following continuous-time dynamics:
\begin{align}\label{eq:projectionXY}
    \dot{x} &= -\nabla_x f(x,y) - \lambda(x,y) \nabla_x h(x,y), \notag \\
    \dot{y} &= -\nabla_y f(x,y) - \lambda(x,y) \nabla_y h(x,y), \\
    \lambda(x,y) &= 
    \frac{\left[-\nabla_x h^\top \nabla_x f - \nabla_y h^\top \nabla_y f + \alpha_{b} (h - \varepsilon^2)\right]_+}
         {\|\nabla_x h\|^2 + \|\nabla_y h\|^2}, \notag
\end{align}
where {we dropped the dependence of $f$ and $h$ on $(x,y)$ for brevity,} \([\,\cdot\,]_+\) denotes projection onto the non-negative reals, and $\lambda$ is the dual variable associated with \eqref{eq: first-order lower-level dynamics 2}. This construction ensures that the trajectories are always feasible, implying that the lower-level solution remains within a controlled suboptimality threshold dictated by \( \varepsilon \). The introduction of this relaxation is motivated by the singularity of the dynamics \eqref{eq:projectionXY} in the limit as \( \varepsilon \to 0 \)~\cite{sharifi2025safe}.

\subsection{Contributions}
To address the limitations of continuous-time solvers in high-dimensional settings, we develop a discrete-time iterative counterpart of the ODE in \eqref{eq:projectionXY}. At each iteration \( k = 0,1,\ldots \), the variables are updated as
\begin{align}
    x^{k+1} &= x^{k} + t^k \Delta x^k, \quad
    y^{k+1} = y^{k} + t^k \Delta y^k, \notag
\end{align}
where \( (\Delta x^k, \Delta y^k) \) are search directions, and \( t^k > 0 \) is the step size. Our framework consists of two key stages: first, computing \( (\Delta x^k, \Delta y^k) \) by solving a convex QCQP, derived as a discrete analog of \eqref{eq: first-order lower-level dynamics 2}; second, performing a novel line search to adaptively select \( t^k \) such that the method ensures both \emph{anytime feasibility}, i.e., \( h(x^k,y^k) \leq \varepsilon^2 \), and \emph{descent}, i.e., \( f(x^{k+1},y^{k+1}) < f(x^k,y^k) \). The QCQP, which we show admits a closed-form solution, is constructed to guarantee that the line search yields a uniformly positive step size, ensuring continued progress toward a stationary point of the relaxed problem \eqref{eq:bilevel-approx-1}. \blue{We establish that the sequence \( (\Delta x^k, \Delta y^k) \) converges to a first-order stationary point of \eqref{eq:bilevel-approx-1} at an \( \mathcal{O}(1/k) \) \emph{ergodic convergence rate} under \emph{local} regularity assumptions.} Overall, our method results in a \emph{tuning-free sequential QCQP} approach for solving \eqref{eq:bilevel-approx-1}, with provable convergence. Finally, we empirically evaluate the performance of the method using numerical experiments.
%

%

%

\section{Main Results}
To streamline notation, we define the concatenated vector \( z = (x,y) \in \mathbb{R}^{n+m} \). We impose the following assumptions.

\begin{assumption}[Upper-level Smoothness]\label{assumption:UL}
$f$ is continuously differentiable; hence $\nabla f$ is locally Lipschitz continuous.
\end{assumption}

\begin{assumption}[Lower-level Smoothness]\label{assumption:LL}
$g$ is twice continuously differentiable; hence $\nabla^2_{yx} g, \nabla^2_{yy} g$ are locally Lipschitz continuous.
\end{assumption}

\begin{assumption}[Lower-level Regularity]\label{assumption:LL-2}
For all $z$ such that $\nabla_y g(z) \neq 0$,
$
    \nabla_y g(z) \notin \mathcal{N}\big( [\nabla^2_{yx} g(z) \  \nabla^2_{yy} g(z)]^\top\big), \notag
$
where $\mathcal{N}(\blue{\cdot})$ denotes the null space. 
\end{assumption}
From the definition \( h(z) = \|\nabla_y g(z)\|_2^2 \), one can verify that under Assumption~\(\ref{assumption:LL-2}\),
$
\nabla h(z) = 0 \iff h(z) = 0.
$


\begin{assumption}[Coercivity]\label{assum:coercive}
$h$ is coercive, i.e., $h(z)=\|\nabla_y g(z)\|^2\to \infty$ as $\|z\|\to \infty$.
\end{assumption}
Coercivity implies that $L_{\varepsilon^2}^{-}(h)$ is bounded for any $\varepsilon \geq 0$. Together with Assumption \ref{assumption:LL}, they imply $\nabla h$ is locally Lipschitz continuous.


\begin{remark}\label{rm:assumption}
Assumption \ref{assumption:LL-2} holds under any of the following:
\begin{enumerate}[leftmargin=*]
    \item The function $y \mapsto g(x,\cdot)$ is strongly convex for every $x$.
    
    \item The matrix $[\nabla^2_{yx} g(\cdot) \  \nabla^2_{yy} g(\cdot)]^\top$ has full rank.
    \item The function $h(\cdot)$ satisfies the PL condition, i.e., 
    \begin{align}
    h(z) - \inf_{z} h(z) =
    h(z) \leq \frac{1}{2\mu} \|\nabla h(z)\|^2 \text{ for } \mu >0. \notag
    \end{align}
\end{enumerate}
\end{remark}
All conditions in \Cref{rm:assumption}, ensure that if $h(z) \neq 0$, then
$\nabla h(z) = [\nabla^2_{yx} g(z) \ \nabla^2_{yy} g(z)]^\top \nabla_yg(z) \neq 0.$ 

\begin{figure}[t]
    \centering
    \includegraphics[width=0.75\linewidth]{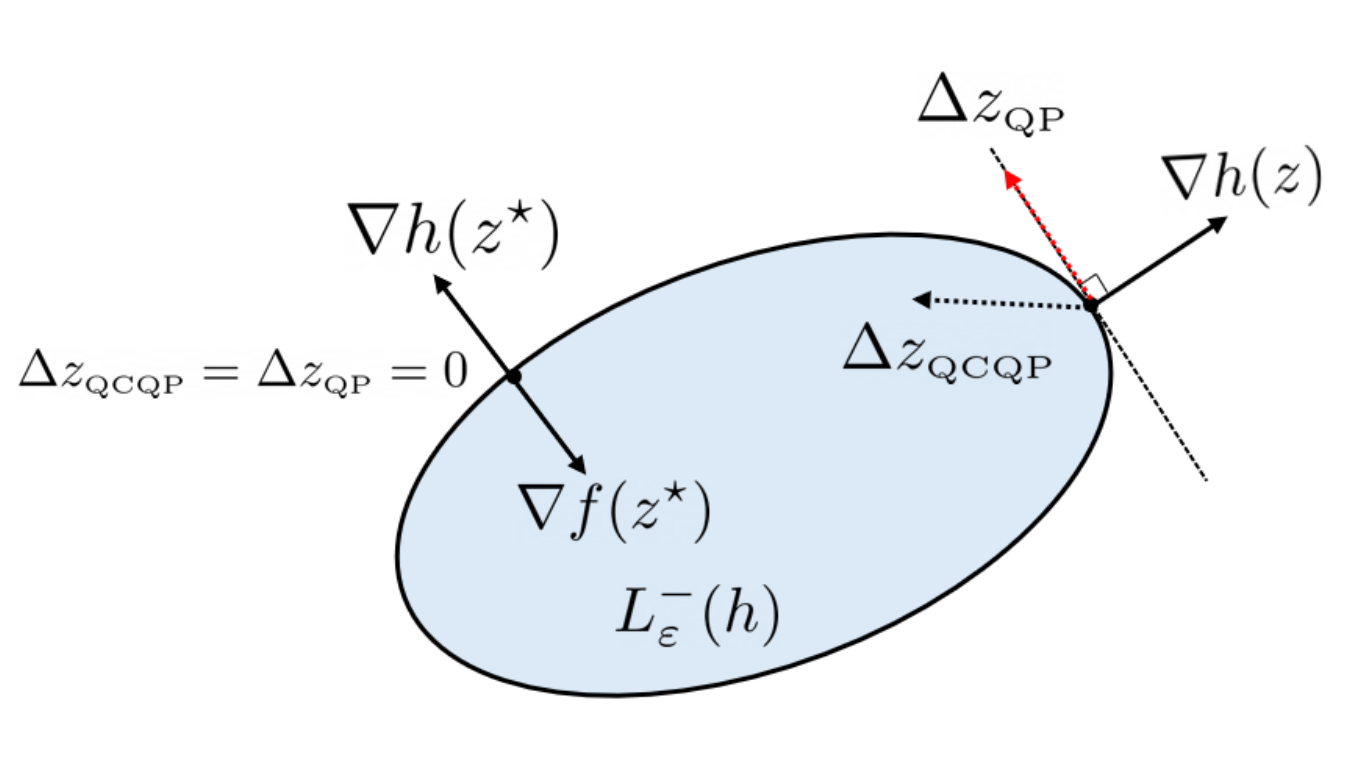}
    \caption{When $h(z) = \varepsilon^2$, the QP may produce an unsafe direction, while the QCQP ensures safety by enforcing $\nabla h(z)^\top \Delta z < 0$.
    }
    \label{fig:example}
\end{figure}
\subsection{Caveat of Search 
Direction Selection via QP}
Suppose the current iterate is feasible, i.e., $z \in L_{\varepsilon^2}^{-}(h)$. A direct discretization of the QP in \eqref{eq: first-order lower-level dynamics 2} leads to the following problem:
\begin{alignat}{2} \label{eq:qp}\tag{QP}
    &\min_{\Delta z}~ \frac{1}{2}\|\Delta z+\nabla f(z)\|^2 \\
    &\text{s.t.} \quad \nabla h(z)^\top \Delta z + \alpha_{b} (h(z) - \varepsilon^2) \leq 0, \notag    
\end{alignat}
where \( \Delta z \) serves as the discrete analog of \( \dot{z} \). However, this approach introduces a fundamental limitation: when the current iterate $z$ approaches the boundary, i.e., as \( h(z) \to \varepsilon^2 \), the admissible step size \( t \) tends to zero. In the limiting case \( h(z) = \varepsilon^2 \), the solution \( \Delta z \) to \eqref{eq:qp} could become \textit{tangent} to the boundary when $\nabla f(z)^\top \nabla h(z) \leq 0$—see \Cref{fig:example} for an illustration. In this case, there exists no strictly positive step size \( t \) such that \( z + t \Delta z \) remains within the feasible set, causing the updates to stagnate. This challenge reflects a broader difficulty in discretizing control barrier functions, which has recently been investigated in \cite{brunke2024practical}. In the following, we propose a step size aware variation of \eqref{eq:qp} to mitigate the issue of vanishingly small step size.

\subsection{Search Direction Selection via QCQP}
Assuming the current iterate is feasible, $z \in L_{\varepsilon^2}^{-}(h)$, we define the following QCQP to select the search direction:
\begin{align} \label{eq:qcqp}\tag{QCQP}
    \min_{\Delta z} \quad & \frac{1}{2} \|\Delta z + \nabla f(z)\|^2 \\
    \text{s.t.} \quad & \nabla h(z)^\top \Delta z + \alpha_{b} (h(z) - \varepsilon^2) \leq - w\|\Delta z\|^2. \notag
\end{align}
where \( w > 0 \) is user-defined tilting parameter. The additional quadratic term \( w\|\Delta z\|^2 \) biases the search direction toward the interior of the feasible set $L_{\varepsilon^2}^{-}(h)$. In particular, on the boundary \( h(z) = \varepsilon^2 \), this term ensures that the QCQP solution satisfies \( \nabla h(z)^\top \Delta z < 0 \) (as long as $\Delta z \neq 0$), i.e., the search direction is strictly inward-pointing—see \Cref{fig:example} for an illustration. This tilting mechanism guarantees that the iterates remain feasible and enables the use of a uniformly positive step size.

Notably, the proposed QCQP admits a closed-form solution provided in Appendix~\ref{sec:appendix}.  The following theorem summarizes its key properties.

\begin{theorem}\label{thm:QCQP}
    Under Assumptions \ref{assumption:UL},\ref{assumption:LL}, and \ref{assumption:LL-2}, the following statements are true for \eqref{eq:qcqp}:
    \begin{enumerate}[label=(\roman*), ref=\thetheorem-\roman*]
        \item \label{thm:qcqp-feas} If $z \in L^{-}_{\varepsilon^2}(h)$, then  \eqref{eq:qcqp} is strictly feasible.

        \item \label{thm:qcqp-descent} 
        If $z \in L^{-}_{\varepsilon^2}(h)$, then $\Delta z$ is a descent direction for $f$,
        \begin{align}\label{eq:descent_direction}
            \nabla f(z)^\top \Delta z \leq - \|\Delta z\|^2.
        \end{align}

        \item If $z \notin L^{-}_{\varepsilon^2}(h)$, and $\Delta z$ exists, then it is a descent direction for $h$,
        \begin{align} \label{eq:descent_direction_for_h}
            \nabla h(z)^\top \Delta z \leq -\alpha_{b} (h(z)-\varepsilon^2) - w\|\Delta z\|^2 \leq 0.
        \end{align}
        \item \label{thm:qcqp-kkt} 
        \eqref{eq:qcqp} returns $\Delta z =0$ when $z$ satisfies the Karush-Kuhn-Tucker (KKT) \cite{nesterov2018lectures} conditions for \eqref{eq:bilevel-approx-1}.
    \end{enumerate}
\end{theorem}
\begin{proof}
    See Appendix \ref{sec:appendix} for the proof.
\end{proof}


\subsection{Step Size Selection: Line Search}
After selecting the update direction \( \Delta z \) via \eqref{eq:qcqp}, we propose a line search procedure that ensures both \emph{convergence} and \emph{anytime feasibility}. The latter is crucial for the former, as maintaining feasibility guarantees that the search direction remains a descent direction for the upper-level objective \( f \) {(See \eqref{eq:descent_direction})}.
This, in turn, allows the use of a backtracking line search to reduce the upper-level objective value at each step.

\medskip
\noindent \textbf{Armijo line search on the upper-level objective.} To ensure convergence, we employ the classical Armijo condition \cite{nesterov2018lectures} combined with a backtracking line search with parameter $\beta \in (0,1)$:
\begin{align}\label{eq:armijo}
    f(z + t \Delta z) \leq f(z) + \alpha_{\mathrm{ls}} t \nabla f(z)^\top \Delta z,
\end{align}
where \( \alpha_{\mathrm{ls}} \in (0,1) \) is a line search parameter. Provided that \( \nabla f(z) ^ \top \Delta z < 0 \) and \( \nabla f \) is locally Lipschitz continuous, we will establish that a uniformly positive step size \( t > 0 \) satisfying \eqref{eq:armijo} always exists (see \Cref{lem:step size}).

\medskip
\noindent \textbf{Safeguarding the step size for anytime feasibility.} The Armijo condition assumes that the iterates remain feasible. To guarantee this, we introduce a safety condition by leveraging discrete-time control barrier functions \cite{agrawal2017discrete, brunke2024practical, cosner2023robust}:
\begin{align}\label{eq:safety}
    h(z + t \Delta z) - \varepsilon^2 \leq (1 - \gamma)\big(h(z) - \varepsilon^2\big),
\end{align}
where \( \gamma \in (0,1) \) is a user-defined parameter that controls the maximum speed of approaching the boundary. This inequality ensures recursive feasibility; specifically, if \( z \in L_{\varepsilon^2}^{-}(h)\), then \( z + t \Delta z \in L_{\varepsilon^2}^{-}(h) \) as well, {ensuring the strict feasibility of the QCQP.} We will prove that, under local Lipschitz continuity of \( \nabla h \), this safeguard yields a uniformly positive step size across iterations (see \Cref{lem:step size}).

Notably, condition \eqref{eq:safety} can be interpreted as a discrete-time analog of the continuous-time barrier condition:
\[
\nabla h(z)^\top \dot{z} + \alpha_{b}(h(z) - \varepsilon^2) \leq 0.
\]
Introducing the ansatz \( \dot{z} \approx \Delta z / \Delta s \), where \( \Delta s \) denotes the discretization time step, and identifying \( \Delta s = t \), we observe that condition \eqref{eq:safety} coincides with the discretized form of the barrier condition when \( \gamma = \alpha_b t \). This correspondence offers further justification for the proposed safeguard, as it faithfully captures the behavior of the continuous-time system in discrete updates. A detailed summary of the complete method is provided in \Cref{alg:method}.
\begin{algorithm}[t]
\caption{Sequential QCQP for BLO with Line Search}\label{alg:method}
\begin{algorithmic}[1]
    \REQUIRE Initialization $z^0 \in \mathbb{R}^{n+m}$, maximum iteration $K > 0$, backtracking parameter $\beta \in (0, 1)$, tolerance parameter $\varepsilon > 0$, maximum step size $t_{\max} > 0$
    \STATE Ensure $h(z^0) \leq \varepsilon^2$ and initialize $z \leftarrow z^0$
    \FOR{$k = 1$ to $K$}
        \STATE Compute $\Delta z^k$ according to \eqref{eq:qcqp}. 
        \STATE $t \gets t_{\max}$
        \STATE \textbf{while} conditions~\eqref{eq:armijo} and~\eqref{eq:safety} are not satisfied: \( t \gets \beta t \)

        \STATE $t^k\gets t$
        \STATE Update $z^{k+1} \gets z^k + t^k \Delta z^k$
    \ENDFOR
    \RETURN $z^K$
\end{algorithmic}
\end{algorithm}

\subsection{Convergence Analysis}
We first prove that the chosen step size $\{t^k\}$ is uniformly bounded away from zero, {if $z^0 \in L_{\varepsilon^2}^{-}(h)$}.
\begin{lemma}[Lower Bound on Step Size]\label{lem:step size}
    The step size $t^k$ chosen by the proposed line search {(\eqref{eq:armijo} and \eqref{eq:safety})} satisfies
    \[
    t^k \geq t_{\min} = \min\{ \beta \frac{2 (1- \alpha_{\mathrm{ls}})}{L_f}, 
     \beta \frac{ \gamma}{\alpha_{b}},
     \beta \frac{2 w}{L_h}, t_{\max} \}, \forall k\geq 1.
    \]
\end{lemma}
\begin{proof}
Define $ \phi(t) = h(z + t \Delta z) - \varepsilon^2 - (1 - \gamma)(h(z) - \varepsilon^2) $, and assume $ \Delta z \neq 0 $. 
If $ \phi(0) = 0 $, then $ h(z) = \varepsilon^2 $, and by Theorem \ref{thm:QCQP}, we have $ \phi'(0) = \nabla h(z)^\top \Delta z < 0 $. By continuity of $h(z)$, there exists $ t > 0 $ such that $ \phi(t) < 0 $, hence $ \phi(t) \leq 0 $. If
$ \phi(0) < 0 $, continuity again implies the existence of $ t > 0 $ such that $ \phi(t) \leq 0 $.

With this argument, there exists a sequence of $t^k > 0$ that satisfies \eqref{eq:safety}, and condition \eqref{eq:safety} 
ensures that $h(z^{k+1})< \varepsilon^2$ (if $h(z^0) < \varepsilon^2$).
This implies that $\{z^k\}_k$ remains bounded due to the coercivity of $h$. Boundedness of $\{z^k\}_k$ implies that for any $k\geq 0$, $\nabla f(z^k)$ and $\nabla h(z^k)$ are Lipschitz with uniform constants $L_f$ and $L_h$, respectively.
%
Using the Descent lemma \cite{nesterov2018lectures} with the Lipschitz continuity of $\nabla f$, a sufficient condition for the Armijo condition \eqref{eq:armijo} to hold is
\begin{align*}  
    &f(z^k + t \Delta z^k) \leq f(z^k) + t^k \nabla f(z^k)^\top \Delta z^k \! + \! \tfrac{L_f(t^k)^2}{2} \|\Delta z^k\|^2 \\
    &\leq f(z^k) + t^k \nabla f(z^k)^\top \Delta z^k + \tfrac{L_f(t^k)^2}{2} (-\nabla f(z^k)^\top\Delta z^k) \notag \\
    &= f(z^k) + t^k (1-\tfrac{L_ft^k}{2}) \nabla f(z^k)^\top \Delta z^k \notag \\
    &\leq f(z^k) + \alpha_{\mathrm{ls}} t^k \nabla f(z^k)^\top\Delta z^k,
\end{align*}
where in the second inequality we have used \eqref{eq:descent_direction}, and the last inequality is satisfied for any $t^k \in [0, \frac{2(1-\alpha_{\mathrm{ls}})}{L_f}]$. 

Now, using the Descent lemma with the Lipschitz continuity of $\nabla h$, a sufficient condition to satisfy \eqref{eq:safety} is
\begin{align*}
    &(h(z^k + t^k \Delta z^k) \!-\! \varepsilon^2) - (h(z^k) \!-\! \varepsilon^2) \\
    &\leq t^k\nabla h(z^k)^\top\Delta z^k +  \tfrac{L_h (t^k)^2}{2} \|\Delta z^k\|^2 
    \\
    &\leq t^k (-\alpha_{b} (h(z^k)\!-\!\varepsilon^2)\!-\! w\|\Delta z^k\|^2) \!+\! \tfrac{L_h (t^k)^2}{2} \|\Delta z^k\|^2 \notag\\
    &= -\alpha_b t^k(h(z^k)-\varepsilon^2) + \big( \tfrac{L_h (t^k)^2}{2} - wt^k \big) \|\Delta z^k\|^2 \\
    &\leq -\gamma (h(z^k)-\varepsilon^2),
\end{align*}
where we have used \eqref{eq:descent_direction_for_h} in the second inequality. The last inequality holds when $t^k \leq \min(\frac{2w}{L_h},\frac{\gamma}{\alpha_{b}})$. Putting these together, a lower bound for the chosen step size is the minimum of $\beta \min(\frac{2(1-\alpha_{\mathrm{ls}})}{L_f},\frac{2w}{L_h},\frac{\gamma}{\alpha_{b}})$ and $t_{\max}$, which concludes the proof.
\end{proof}

\begin{remark}
    Lemma \ref{lem:step size} highlights the necessity of introducing the tilting parameter \( w > 0 \) in the search direction optimization problem, as it ensures a strictly positive lower bound on the step size and thereby prevents the algorithm from halting.    
\end{remark}
%
\blue{Next, we prove the anytime feasibility of Algorithm \ref{alg:method} for problem \eqref{eq:bilevel-approx-1}, as well as ergodic convergence of the sequence $\{\|\Delta z^k\|^2\}_k$.}

\begin{theorem}\label{thm:convergence}
    Assume that $z^0 \in L_{\varepsilon^2}^{-}(h)$.
    Consider the iterative algorithm outlined in \Cref{alg:method}. Then for all $K=0,1,\ldots$,
    $$
    (h(z^K)-\varepsilon^2) \leq (1-\gamma)^K (h(z^0)-\varepsilon^2) \leq 0.
    $$
    \blue{Furthermore, for all $K=1,2,\cdots$, we have
    $$
    \frac{1}{K}\sum_{k=0}^{K-1} \|\Delta z^k\|^2 \leq \frac{f(z^0) - f^\star_\varepsilon}{\alpha_{\mathrm{ls}}t_{\min}K}.
    $$
    In particular, $\lim_{k \to \infty} \|\Delta z^k\|=0$.}
\end{theorem}
\begin{proof}
From the safety line search \eqref{eq:safety}, we have  
\begin{equation}
    h(z^k + t^k \Delta z^k) - \varepsilon^2 \leq
    (1 - \gamma) (h(z^k)-\varepsilon^2) \leq 0.\notag
\end{equation}
Similarly, from the Armijo line search \eqref{eq:armijo}, it follows that  
\begin{align}
    f(z^{k+1}) 
    &\leq f(z^k) + \alpha_{\mathrm{ls}} t^k \nabla f(z^k)^\top\Delta z^k \notag\\
    &\leq f(z^k) - \alpha_{\mathrm{ls}} t^k \|\Delta z^k\|^2 \notag \\
    &\leq f(z^k) - \alpha_{\mathrm{ls}} t_{\min} \|\Delta z^k\|^2, \notag
\end{align}
where in the second inequality we used \eqref{eq:descent_direction} and in the final inequality we used the result of \Cref{lem:step size}.
\blue{
Next, we obtain a telescoping sum by adding both sides of the inequality for $k=0,\ldots,K-1$:
$$
\sum_{k=0}^{K-1} \alpha_{\mathrm{ls}} t_{\min}\|\Delta z^k\|^2 \leq f(z^0)-f(z^{K}) \leq f(z^0)-f^\star_\varepsilon
$$
where the second inequality follows from the fact that $z^{K} \in L_{\varepsilon^2}^{-}(h)$. We then arrive at
$$
\sum_{k=0}^{K-1} \|\Delta z^k\|^2 \leq \frac{f(z^0) - f^\star_\varepsilon}{\alpha_{\mathrm{ls}}  t_{\min}}<\infty.
$$
Dividing by $K$ leads to the desired inequality. Furthermore, every term in the series is non-negative, and the series is uniformly upper-bounded; we conclude that \( \lim_{k \to \infty} \|\Delta z^k\|^2 = 0 \).
Finally, using the continuity of the square root function, we have \( \lim_{k \to \infty} \|\Delta z^k\| = 0 \).
}  
\end{proof}

\begin{figure*}[t]
    \centering
    \includegraphics[width=0.3\linewidth]{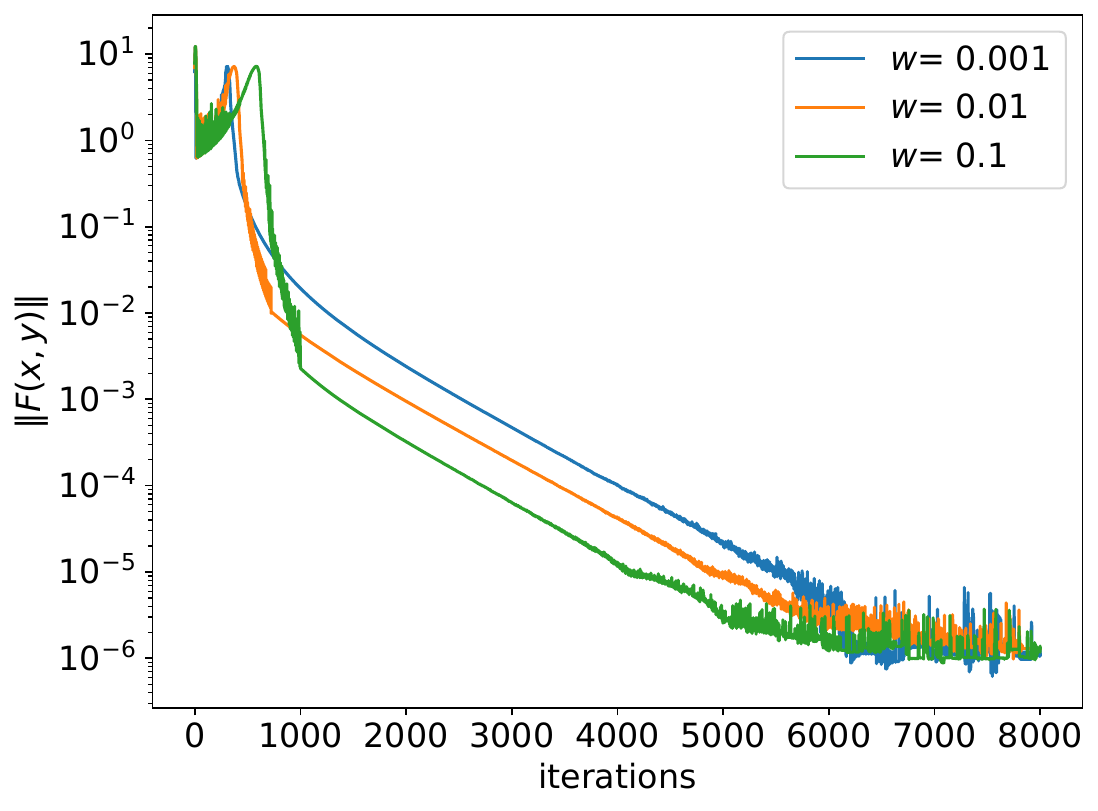}
    \includegraphics[width=0.3\linewidth]{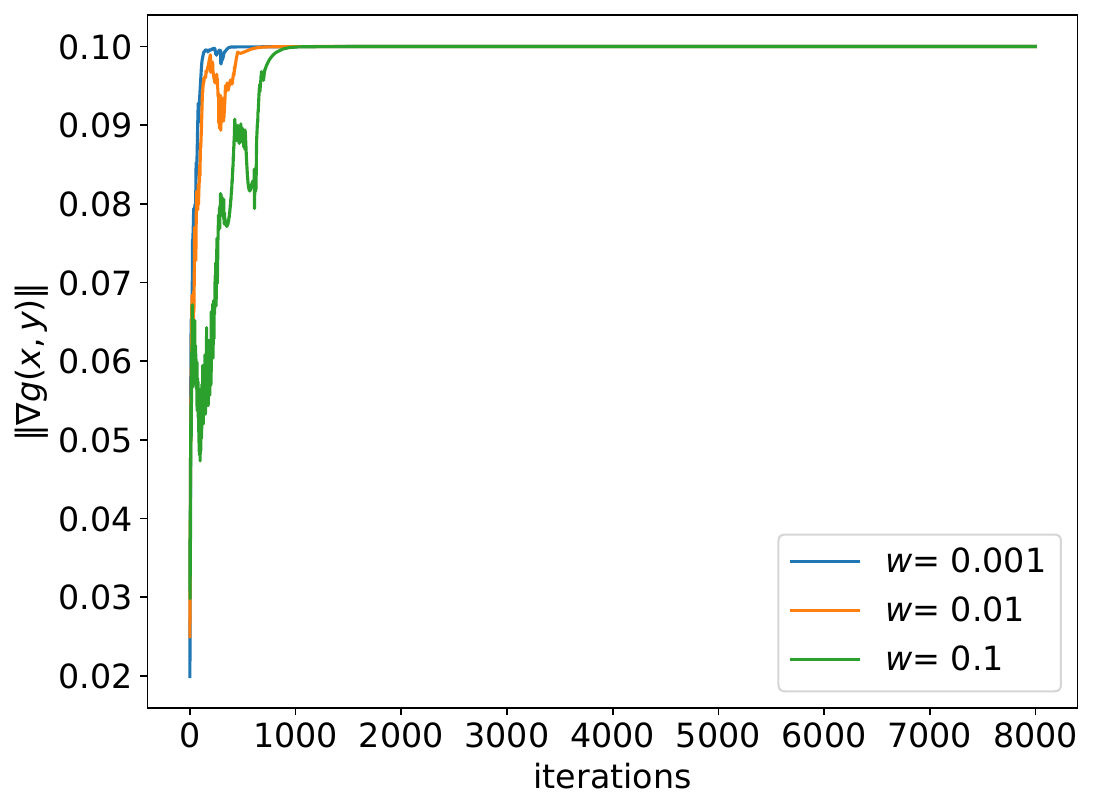}
    \includegraphics[width=0.3\linewidth]{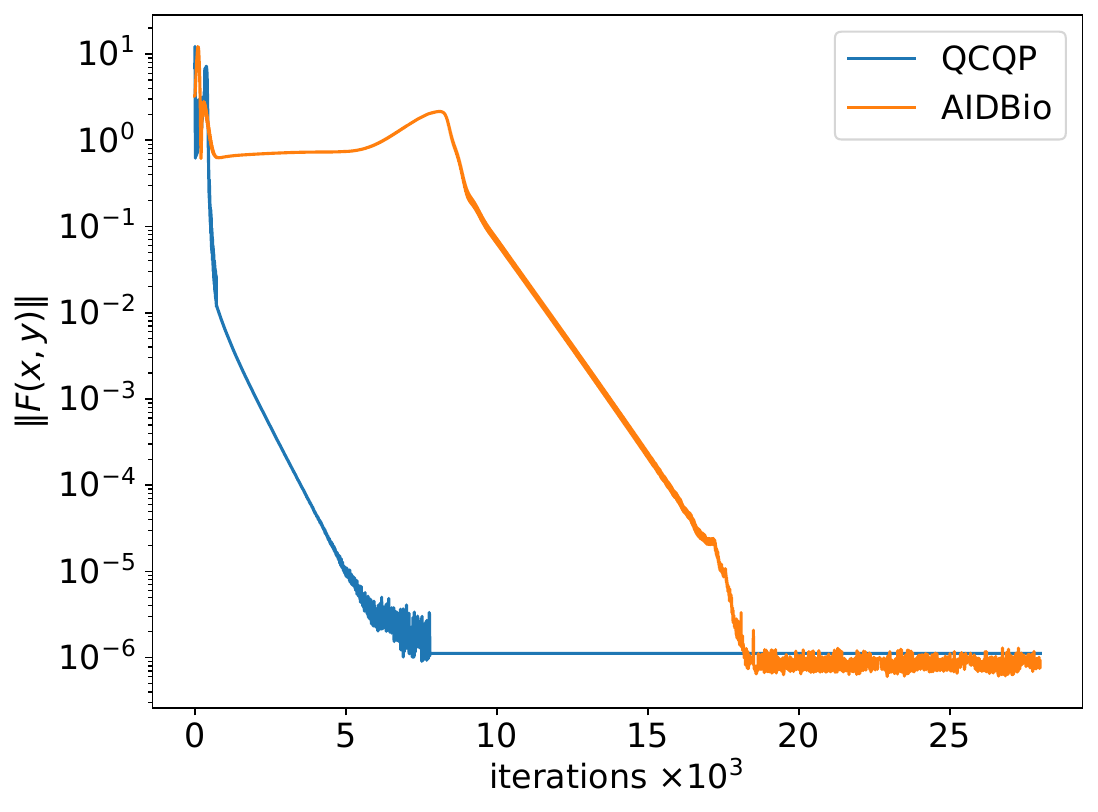}
    \caption{Synthetic examples: Ablation study on the effect of $w$: (left) the norm of hyper-gradient, (middle) the lower-level suboptimality. 
    (right) Comparison with AIDBio with $w = 0.01$.}
    \label{fig:w_ablation}
\end{figure*}

\begin{table}[t]
    \centering
    \resizebox{0.7\linewidth}{!}{ 
        \begin{tabular}{l c c c}
            \toprule
            \textbf{Method} & \textbf{LL} & \textbf{Loops} & \textbf{Complexity}  \\
            \midrule
            AIDBio & SC ($g$)  & Double & \(\mathcal{O}(\epsilon^{-1})\) \\
            \midrule
            BOME   & PL ($g$) & Double & \(\mathcal{O}(\epsilon^{-3})\) \\
            \midrule
            VPBGD  & PL ($g$)  & Double & \(\mathcal{O}(\epsilon^{-3/2})\) \\
            \midrule
            Ours   & PL ($h$)    & Single & \(\mathcal{O}(\epsilon^{-1})\) \\
            \bottomrule
        \end{tabular}
    }
    \caption{
    \small Comparison of methods in the experiments. “LL” denotes assumptions on the lower-level objective, with “SC” and “PL” indicating strong convexity and PL condition, respectively. Complexity refers to finding an \( \epsilon \)-stationary point.}
    \label{tab:compare}
\end{table}

\section{Experiments}
In this section, we present a numerical evaluation of our approach by benchmarking it against other methods (AIDBio \cite{ji2021bilevel}, BOME \cite{liu2022bome}, and VPBGD \cite{shen2023penalty}) on both a synthetic example and a data hyper-cleaning (DHC) task on the MNIST \cite{lecun2010mnist} dataset. See \Cref{tab:compare} for a concise comparison of these methods.
For the synthetic example, we compare the norm of the estimation of the hyper-gradient ($\|F(x,y)\|$) and the lower-level sub-optimality ($\|\nabla_y g(x,y)\|$) as the metrics for convergence and safety.
For the DHC problems, we plot the validation loss and test accuracy, and the norm of the step $\|\Delta z\|$ as the metric.
For all the experiments, we choose the line search parameters as $\beta = 0.5$, $\alpha_{\mathrm{ls}} = \gamma = 0.1$, and $t_{\max} = 1$.
We also choose $w \in \{10^{-1},10^{-2},10^{-3}\}$ and $\alpha_b = 0.1$. We choose $\varepsilon = 0.5$ for the neural network experiment and $\varepsilon = 0.1$ for the rest.
The code is available at \href{https://github.com/o4lc/SGF-BLO}{https://github.com/o4lc/SGF-BLO}.

\subsection{Synthetic Example}
Following \cite{sharifi2025safe}, we consider the following basic bilevel problem 
\begin{align}
        &\min_x \quad  \sin(c^\top x + d^\top y^\star(x)) +  \log(\|x+y^\star(x)\|^2 + 1) \notag \\ 
        &\text{s.t.} \quad  y^\star(x) \in \argmin_y  \tfrac{1}{2} \|Hy - x\|^2,\notag
\end{align}
where $x, y, c, d \in \mathbb{R}^{20}$ and $H \in \mathbb{R}^{20 \times 20}$ is randomly generated such that its condition number is no larger than 10.
We use this example as a proof of concept and to conduct ablation studies on the parameter. \Cref{fig:w_ablation} shows a comparison with AIDBio and the effect of the hyper‐parameter \(w\) on the method's convergence. 
A larger \(w\) places more weight on the lower-level constraint, slowing the convergence of \(\|\nabla_y g(\cdot)\|\) to \(\varepsilon\).

\subsection{Data Hyper-Cleaning (DHC)}
We consider the DHC task \cite{franceschi2017forward, shaban2019truncated, shen2023penalty} as follows:
\begin{align}
    &\min_x \frac{1}{N_{\text{val}}} \sum_{(a_{i}, b_{i}) \in \mathcal{D}_{\text{val}}} \mathcal{L}(f_{y^\star(x)} (a_{i}) , b_{i}) \notag \\
    &\text{s.t.} y^\star(x) \! \in \! \argmin_y
    \frac{1}{N_{\text{tr}}} 
    \sum_{(a_{i}, b_{i}) \in \mathcal{D}_{\text{tr}}} \! \sigma(x_i) \mathcal{L}(f_{y} (a_{i}), b_{i}) \! + \! \lambda  \|y\|^2\!, \notag
\end{align}
where $\lambda$, $\sigma(\blue{\cdot})$, and $\mathcal{L}(\blue{\cdot})$ represent the regularizer, sigmoid function, and cross-entropy loss, respectively, and $f_y(\blue{\cdot})$ denotes a classifier parameterized by $y$.
Furthermore, $(a, b) \in \mathcal{D}$ denotes the data and labels in dataset $\mathcal{D}$.

We experiment with two choices of \(f_y\): one using a linear classifier (convex) and the other using a neural network (non-convex) with one hidden layer with 50 neurons and ReLU activation.
\Cref{fig:DHC} and \ref{fig:nn} compare our method with BOME and VPBGD in terms of validation loss, test accuracy, and the norm of the chosen step for different corruption rates on the linear classifier and neural network classifier, respectively.


\begin{figure}[t]
    \centering
    \includegraphics[width=0.49\linewidth]{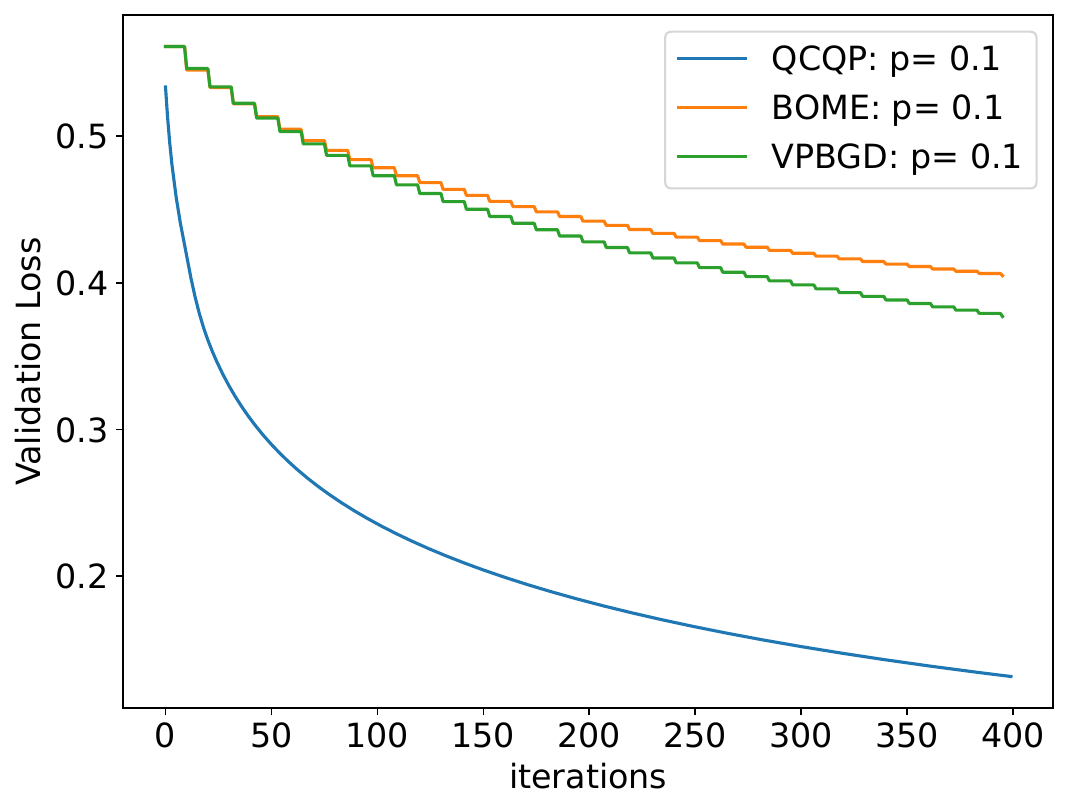}
    \includegraphics[width=0.49\linewidth]{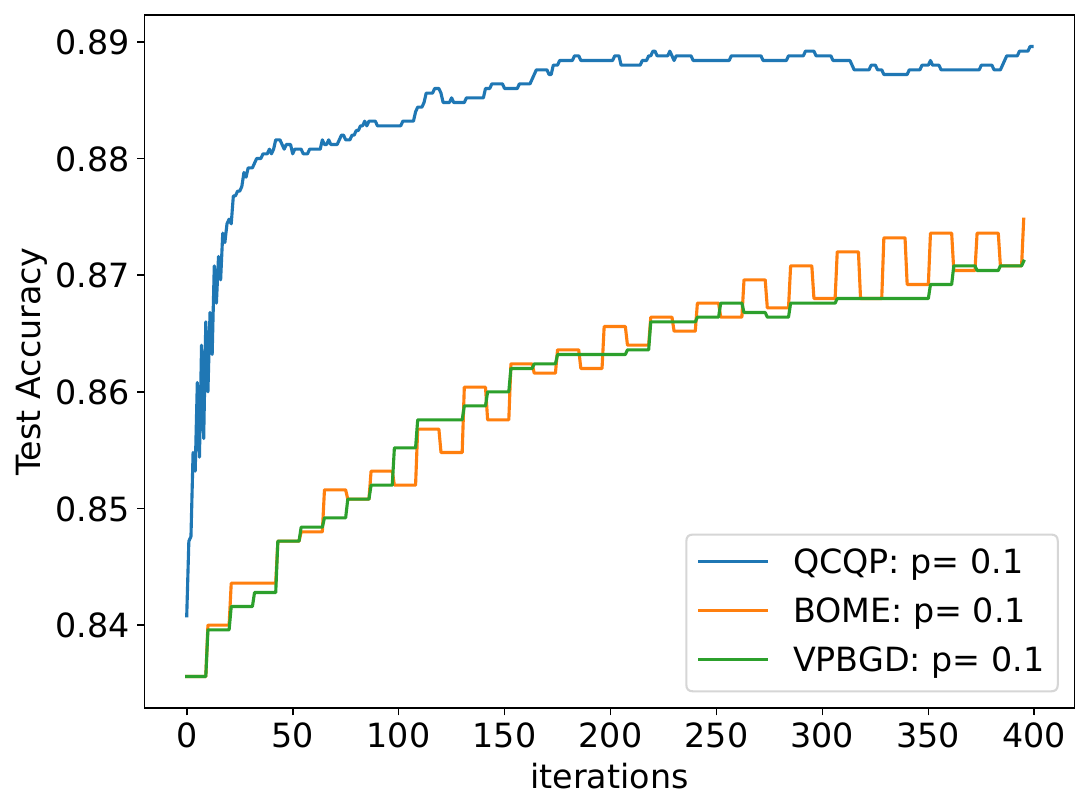}
    
    \includegraphics[width=0.49\linewidth]{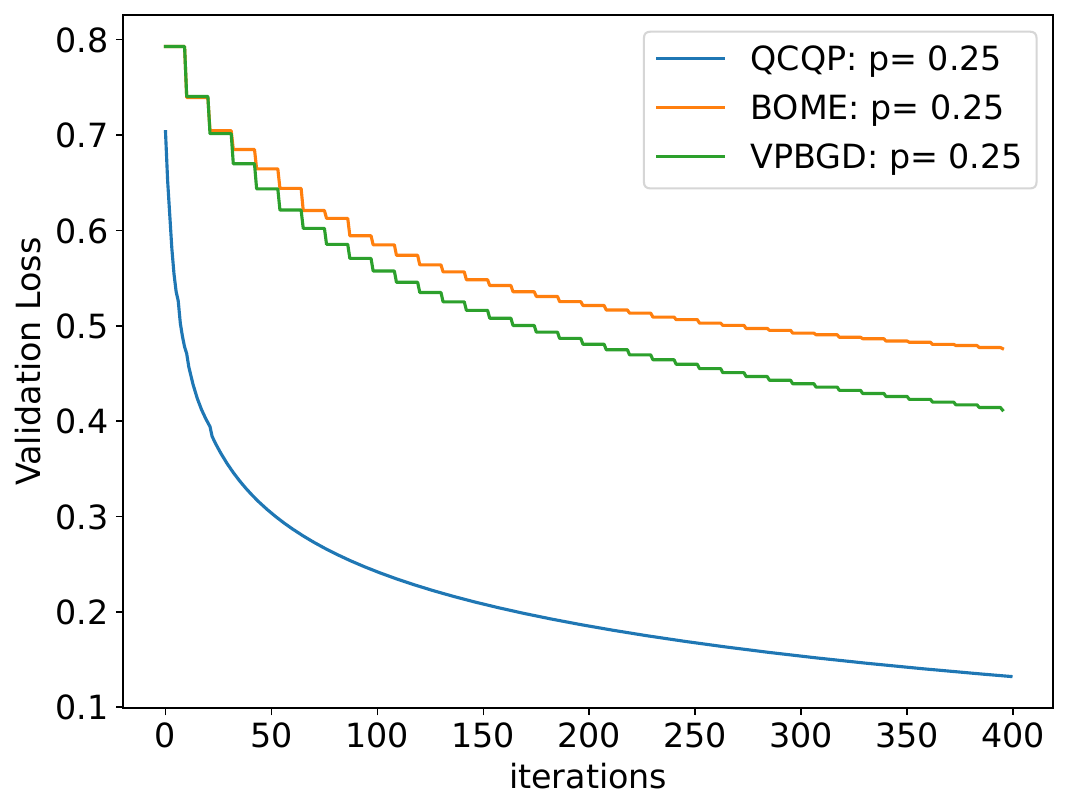}
    \includegraphics[width=0.49\linewidth]{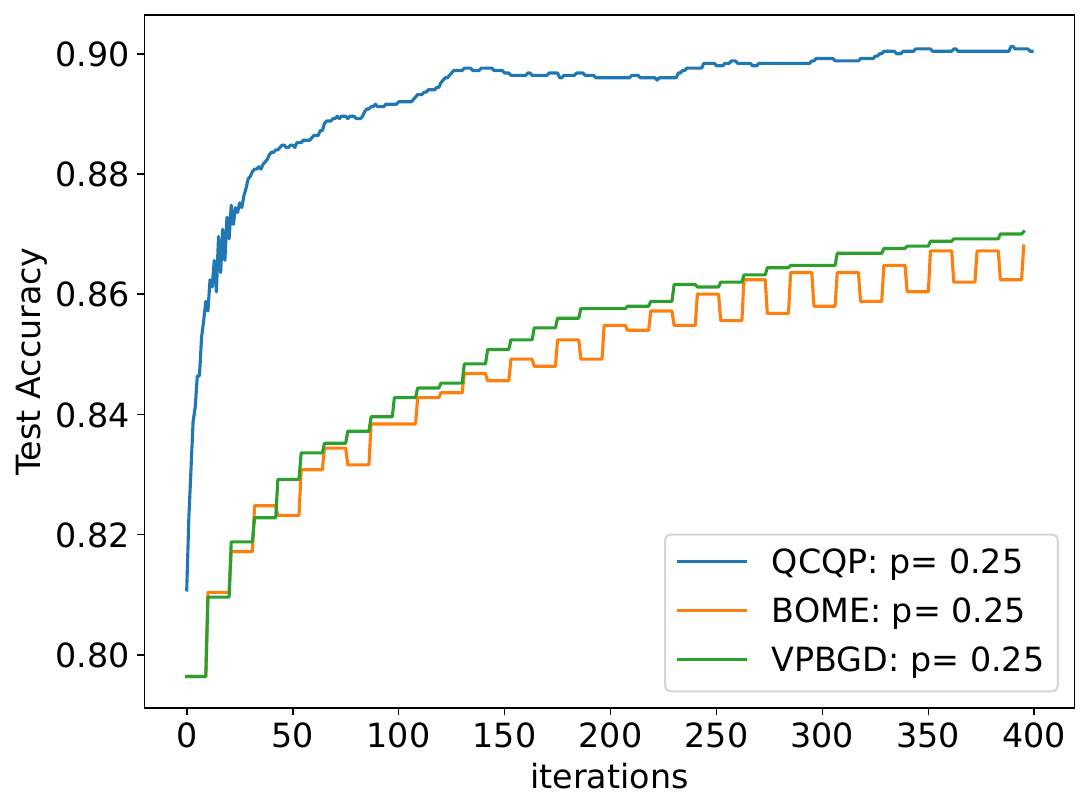}
    \caption{DHC task with a linear classifier: Comparison of validation loss and test accuracy for $p=\{10\%, 25\%\}$.}
    \label{fig:DHC}
\end{figure}

\begin{figure*}[t]
    \centering
    \includegraphics[width=0.3\linewidth]{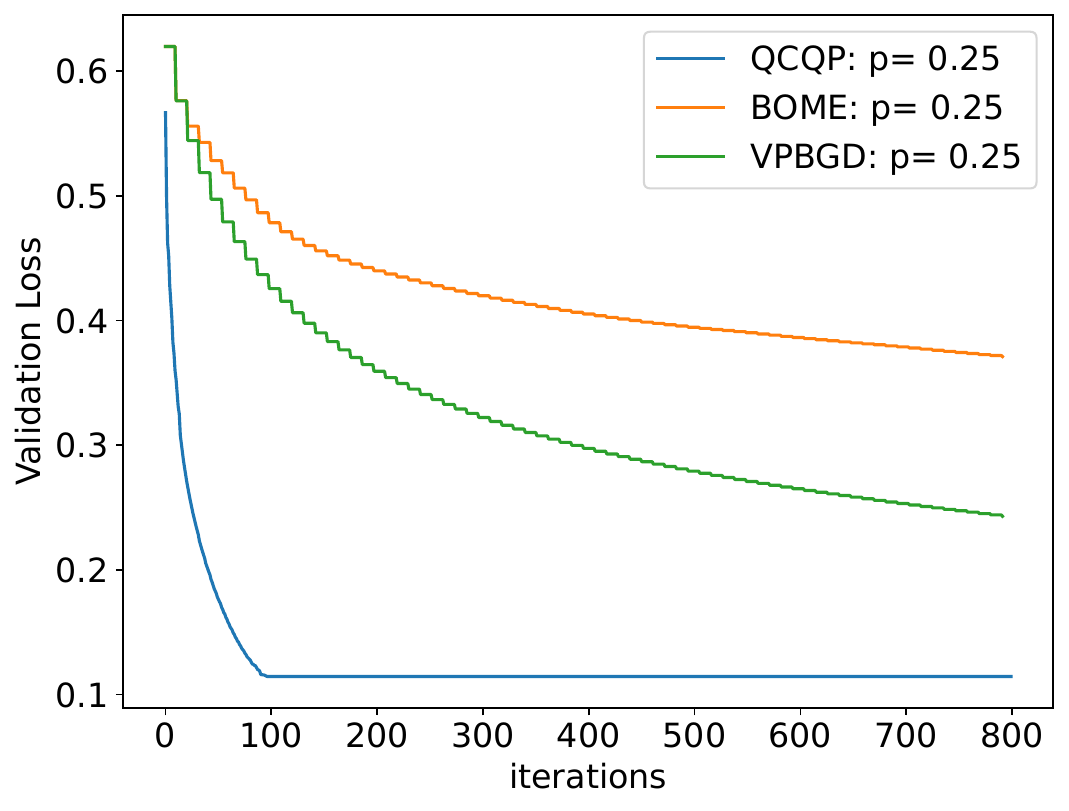}
    \includegraphics[width=0.3\linewidth]{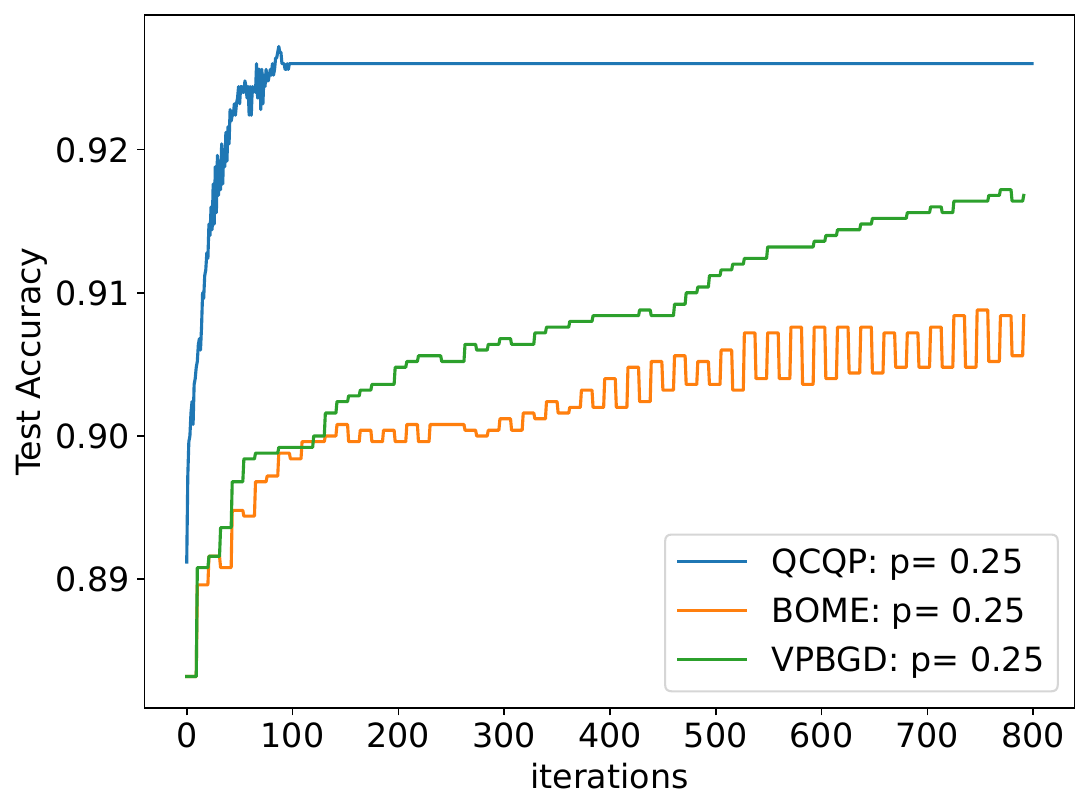}
    \includegraphics[width=0.3\linewidth]{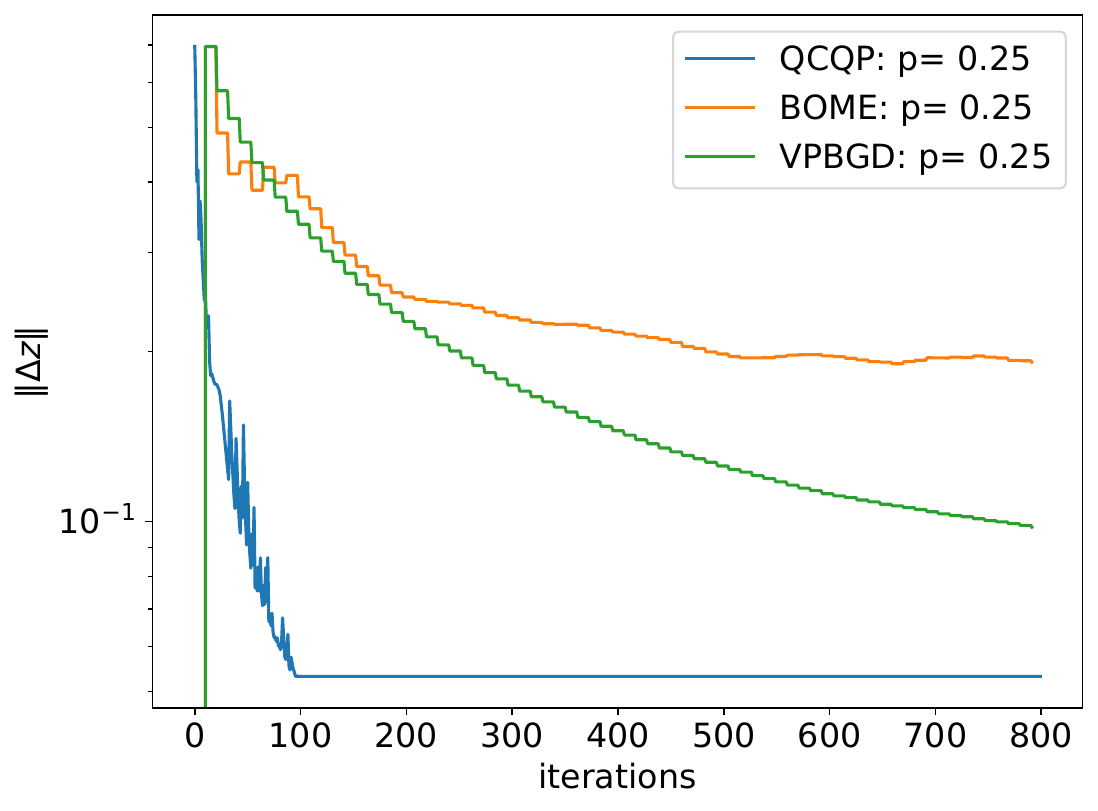}
     \caption{DHC task with neural network classifier: Comparison of validation loss and test accuracy and $\|\Delta z\|$ for $p=25\%$.}
    \label{fig:nn}
\end{figure*}

\section{Conclusion}

We presented a novel, tuning-free algorithm for bilevel optimization by discretizing a continuous-time framework based on control barrier functions. The method operates in a single loop, computes search directions via sequential convex QCQPs, and adaptively selects step sizes through a line search that ensures anytime feasibility and descent. We proved an ergodic convergence rate of $\mathcal{O}(1/k)$ under mild regularity assumptions and demonstrated the practical effectiveness of the method through numerical experiments. Potential future directions for our approach include investigating stochastic settings and incorporating lower-level constraints.



\section{Appendix}\label{sec:appendix}


\subsection{Derivations and Proofs}
\begin{remark}\label{rm:qcqpsol}
    The solution of \eqref{eq:qcqp} can be computed by solving the following equivalent problem
    \begin{align}
        \min_{\Delta z} ~~& \frac{1}{2}\|\Delta z + \nabla f(z)\|^2 \notag \\
        s.t. ~~& \|\Delta z + \frac{\nabla h(z)}{2w}\|^2 \leq \|\frac{\nabla h(z)}{2w}\|^2 - \frac{\alpha_b}{w}(h(z) - \varepsilon^2) \notag
    \end{align}
    which is the projection of $(-\nabla f(z))$ on a ball. This projection has a closed-form solution of the following form 
    \begin{align}
        \Delta z =
        \begin{cases}
         -\nabla f(z) &\text{if } \|-\nabla f(z) - c\|^2 \leq r^2 \\
          c + r \frac{-\nabla f(z) - c}{\|-\nabla f(z) - c\|} &\text{else}
        \end{cases}\notag
    \end{align}
    where $c = -\frac{\nabla h(z)}{2w}$, and $r = \sqrt{\|\frac{\nabla h(z)}{2w}\|^2 - \frac{\alpha_b}{w}(h(z) - \varepsilon^2)}$.
    %
\end{remark}
{This problem is feasible when $r \geq 0$.
Moreover, if \( r \geq 0 \), the resulting \( \Delta z \) is norm-bounded.
This is because \( \Delta z \) is either equal to \( -\nabla f(z) \) or lies on the ball defined by \( c \) and \( r \).
}

\smallskip \noindent
\textbf{Proof of \Cref{thm:QCQP}:}
The KKT conditions for \ref{eq:qcqp} are
\begin{align}\label{eq:KKTQCQP}
\begin{cases}
    \Delta z + \nabla f(z) + \lambda \nabla h(z) + 2 \lambda w \Delta z = 0, \\
    \nabla h(z)^\top \Delta z \leq -\alpha_b (h(z)-\varepsilon^2) - w\|\Delta z\|^2,\\
    \lambda \geq 0, \\
    \lambda \big( 
        \nabla h(z)^\top \Delta z +\alpha_b (h(z)-\varepsilon^2) + w\|\Delta z\|^2 \big) = 0, 
\end{cases}
\end{align}
where $\lambda \geq 0$ is the Lagrangian multiplier. 

\begin{proof}[\ref{thm:qcqp-feas}]
    The constraint of  \eqref{eq:qcqp} is minimized at 
    $\Delta z = \frac{-\nabla h(z)}{2w}$ and its minimum is
    $
    \frac{-1}{4w}\|\nabla h(z)\|^2 + \alpha_b (h(z)-\varepsilon^2).
    $
    If $h(z)<\varepsilon^2$, the minimum is trivially negative. 
    If $h(z)=\varepsilon^2$, the minimum is still negative, i.e., \eqref{eq:qcqp} is strictly feasible, because $\nabla h(z) \neq 0$, according to Assumption \ref{assumption:LL-2}.
    According to Slater's condition \cite{nesterov2018lectures}, this also implies strong duality of \eqref{eq:qcqp}.
\end{proof}


\begin{proof}[\ref{thm:qcqp-descent}]
    Using \eqref{eq:KKTQCQP}, the KKT conditions of the \eqref{eq:qcqp}, we have
    \begin{align} 
        &\nabla f(z)^\top \Delta z = - \|\Delta z\|^2 - \lambda \nabla h(z)^\top \Delta z - 2 \lambda w \|\Delta z\|^2  \notag \\
            &= - \|\Delta z\|^2 + \lambda \alpha_b (h(z)-\varepsilon^2) + \lambda w \|\Delta z\|^2 - 2 \lambda w \|\Delta z\|^2 \notag \\
            &= - \|\Delta z\|^2 + \lambda \alpha_b (h(z)-\varepsilon^2) - \lambda w \|\Delta z\|^2 \notag \\
            &\leq - \|\Delta z\|^2 - \lambda w \|\Delta z\|^2 \leq - \|\Delta z\|^2, \notag
    \end{align}    
    where in the penultimate inequality we used $z \in L^{-}_{\varepsilon^2}(h)$.
\end{proof}


\begin{proof}[\ref{thm:qcqp-kkt}]
At $\Delta z = 0$, \eqref{eq:KKTQCQP} reduces to the KKT conditions of \eqref{eq:bilevel-approx-1}:
\begin{align}
    \begin{cases}
        \nabla f(z) + \lambda \nabla h(z) = 0, \\
        h(z) - \varepsilon^2 \leq 0,~ 
        \lambda \geq 0, \\
        \lambda (h(z) - \varepsilon^2) = 0 .
    \end{cases}\notag
\end{align}
%
\vspace{-1.2cm}
\end{proof}

\phantom{\cite{auslender2010moving, birgin2018augmented, abolfazli2025perturbed}}
\section{Extended Appendix}
\subsection{Additional Remarks}
\begin{remark}\label{rm:MFCQ}
    The following assumptions may be used as alternatives to Assumption \eqref{assumption:LL-2}, as they are sufficient for the purposes of our analysis.
    \begin{enumerate}
        \item The Mangasarian-Fromovitz Constraint Qualification (MFCQ) holds for \eqref{eq:bilevel-approx-1}:
    $\exists \Delta z ~s.t.~ \nabla h(z)^\top \Delta z < 0$ when $h(z)=\varepsilon^2$. 
    \item The Linear Independence Constraint Qualification (LICQ) holds for \eqref{eq:bilevel-approx-1}: $\nabla h(z)\neq 0$ when $h(z)=\varepsilon^2$.
    \end{enumerate}
\end{remark}

Note that LICQ and MFCQ are equivalent if the constraint is scalar; this can be verified from their definition.
Moreover, Assumption \ref{assumption:LL-2} is a stronger assumption than \Cref{rm:MFCQ} because Assumption \ref{assumption:LL-2} is equivalent to LICQ/MFCQ holding for \eqref{eq:bilevel-approx-1} for all $\varepsilon > 0$.
Moreover, according to Assumption \ref{assumption:LL-2}, or equivalently \Cref{rm:MFCQ}, strong duality holds for \eqref{eq:bilevel-approx-1}.

\begin{remark}[Comparison with MBA \cite{auslender2010moving}]
    The Moving Ball Algorithm (MBA) is a method for solving smooth constrained minimization problems.
    A key difference between MBA and \Cref{alg:method} is that the MBA assumes the knowledge of the Lipschitz constants of the gradients and uses the descent lemma to upper bound each constraint and then solves a QCQP to find the next iterate directly.
    Our QCQP, on the other hand, makes no assumptions regarding the Liptchitz constant and performs a line search to ensure anytime feasibility.
\end{remark}

\subsection{Relationship between Bilevel Formulations}
We begin by stating the following definition, which will be central to the analysis of the relationship between the solutions of different bilevel formulations.
Consider an optimization problem of the form:
\begin{align}\label{eq:genericOP}
    \min_x f(x) \quad \text{s.t.}\quad  h(x) = 0,\  g(x) \leq 0, 
\end{align}
where $f: \mathbb{R}^n \mapsto \mathbb{R}$, $h: \mathbb{R}^n \mapsto \mathbb{R}^m$ , and $g: \mathbb{R}^n \mapsto \mathbb{R}^p$.
\begin{definition}[${\epsilon}$-KKT Points \cite{birgin2018augmented}]\label{def:epsKKT}
Given ${\epsilon} \geq 0$, we say that $x\in \mathbb{R}^n$ is an ${\epsilon}$-KKT for the problem \eqref{eq:genericOP} if there exists $\lambda \in \mathbb{R}^m$ and $\mu \in \mathbb{R}_{+}^p$ such that
$$
\begin{cases}
    \|\nabla f(x) + \sum_{i=1}^m \lambda_i \nabla h_i(x) + \sum_{i=1}^p \mu_i \nabla g_{i}(x)\| \leq {\epsilon}, \\
    \|h(x)\| \leq {\epsilon}, \|g(x)\| \leq {\epsilon}, \\
    \mu_i =0 \text{ whenever } g_i(x) < -{\epsilon}, \ i = 1, \ldots, p.
\end{cases}
$$
\end{definition}
Note that with ${\epsilon} = 0$, the KKT conditions are recovered. 



Consider the stationary-seeking reformulation of \eqref{eq:BLO},
\begin{align}\label{eq:BLO_station}
    \min_{x \in \mathbb{R}^n,\, y\in \mathbb{R}^m} \; f(x,y) \quad
    \text{s.t.} \quad \nabla_y g(x, y) = 0.
\end{align}
Define $H = [\nabla^2_{yx} g(x,y) \ \nabla^2_{yy} g(x,y)]$. The $\tilde{\epsilon}$-KKT conditions of \eqref{eq:BLO_station} are 
    $\|\nabla f(z) + H^\top \nu\| \leq \tilde{\epsilon}, ~
    \|\nabla_y g(z)\| \leq \tilde{\epsilon}$.  
To avoid handling multiple constraints, \cite{sharifi2025safe} reformulates \eqref{eq:BLO_station} so the new formulation has only a single constraint,
\begin{align}\label{eq:BLO_station2}
    \min_{x \in \mathbb{R}^n,\, y\in \mathbb{R}^m} \; f(x,y) \quad
    \text{s.t.} \ \|\nabla_y g(x, y)\|^2 = 0.
\end{align}
A downside of this reduction is that by definition, \eqref{eq:BLO_station2} cannot satisfy either LICQ or MFCQ, because whenever the constraint is active, i.e., $\nabla_y g(x,y) = 0$, the gradient of the constraint, $[\nabla^2_{yx} g(x,y) \ \nabla^2_{yy} g(x,y)]^\top \nabla_y g(x,y)$ is also zero.
However, \cite{abolfazli2025perturbed} shows that the $\tilde{\epsilon}$-KKT points of \eqref{eq:BLO_station2} also satisfy $\tilde{\epsilon}$-KKT conditions of \eqref{eq:BLO_station} with a different choice of dual variable. 

Now consider the relaxed formulation in \eqref{eq:bilevel-approx-1}. According to \Cref{def:epsKKT}, the $\tilde{\epsilon}$-KKT conditions of \eqref{eq:bilevel-approx-1} are as follows:
\[
\begin{cases}
    \|\nabla f(z) + \lambda H^\top \nabla_y g(z)\| \leq \tilde{\epsilon}, \\
    \|\nabla_y g(z)\|^2 - \varepsilon^2 \leq \tilde{\epsilon},~
    \lambda \geq 0, \\
    \lambda =0 \text{ whenever } \|\nabla_y g(z)\|^2 - \varepsilon^2 < -\tilde{\epsilon},
\end{cases}
\]

With this definition, every $\tilde{\epsilon}$-KKT point of \eqref{eq:bilevel-approx-1} is also an $\max(\tilde{\epsilon}, \sqrt{\varepsilon^2 + \tilde{\epsilon}})$-KKT point of \eqref{eq:BLO_station} with $\nu = \lambda \nabla_y g(z)$.
In particular, the KKT points of \eqref{eq:bilevel-approx-1} (i.e., when $ \tilde{\epsilon}=0$) are $\varepsilon$-KKT points of \eqref{eq:BLO_station}.

\vspace{-3mm}

\subsection{More on \Cref{thm:convergence}}
\blue{
Let $G(z) \colon z \mapsto \Delta z$ be the map defined by \eqref{eq:qcqp}. The exact formulation for this map is provided in \Cref{rm:qcqpsol}. We first prove that $G(\cdot)$ is continuous. 

\begin{proposition}\label{prop:map_cont}
    The map $G(\cdot)$ is continuous in $z$.
\end{proposition}
\begin{proof}
    According to \Cref{rm:qcqpsol}, map $G(\cdot)$ is a piecewise function, and to show the continuity of this map, we will show that the sub-functions are continuous on the corresponding intervals and there is no discontinuity at an endpoint of any subdomain within that interval \cite{abbott2015understanding}.
    According to Assumption \ref{assumption:UL}, $\Delta z = -\nabla f(z)$ is continuous in $z$. 
    This is also the case for
    $\Delta z = c + r \frac{-\nabla f(z) - c}{\|-\nabla f(z) - c\|}$, because as long as $r > 0$, which is ensured by feasibility of \eqref{eq:qcqp}, $\|-\nabla f(z) -c\| > 0$.
    Finally, whenever $\|-\nabla f(z) - c\| = r$, we have that 
    $$
    c + r \frac{-\nabla f(z) - c}{\|-\nabla f(z) - c\|} = -\nabla f(z) ,
    $$
    which concludes the proof.
\end{proof}


\begin{proposition}
    The limit point of any convergent subsequence of $\{z^k\}$ is a KKT point of problem \eqref{eq:bilevel-approx-1}.
\end{proposition}
\begin{proof}
    From the safety line search \eqref{eq:safety}, we have  
    \begin{equation}
        h(z^k + t^k \Delta z^k) - \varepsilon^2 \leq
        (1 - \gamma) (h(z^k)-\varepsilon^2) \leq 0,\notag
    \end{equation} 
    which implies that $h(z^{k+1})< \varepsilon^2$ for any $k\geq 0$. Therefore, from Assumption \ref{assum:coercive}, we conclude that $\{z^k\}_{k\geq 0}$ is a bounded sequence in $\mathbb{R}^{n+m}$.
    Based on the Bolzano-Weierstrass theorem \cite{bartle2000introduction}, $\{z^{k}\}_k$ has at least one convergent subsequence. Denote this subsequence by $\{z^{k_j}\}_j$ with the limit point $\bar{z}$. From \Cref{thm:convergence}, we have $\lim_{j\to \infty} G(z^{k_j}) = 0$. 
    Now, using the continuity of $G(\cdot)$ from \Cref{prop:map_cont}, we conclude that $$0=\lim_{j\to \infty} G(z^{k_j}) = G(\lim_{j\to \infty} z^{k_j}) = G(\bar{z}).$$
    Therefore, by invoking \Cref{thm:qcqp-kkt}, we show that $\bar{z}$ satisfies the KKT condition of problem \eqref{eq:bilevel-approx-1}.
\end{proof}
}
\vspace{-5mm}

\bibliographystyle{ieeetr}
\bibliography{bib}

\end{document}